\documentclass[12pt,a4paper,reqno]{article}
\usepackage[colorlinks=true,linkcolor=blue,citecolor=magenta]{hyperref}
\usepackage[nottoc,numbib]{tocbibind}
\usepackage[margin=1in]{geometry}
\usepackage{enumitem}
\usepackage{graphicx,algorithm,algorithmic}
\usepackage{amsfonts}
\usepackage{amssymb}
\usepackage{amsmath,amsthm}
\usepackage{longtable}
\usepackage{rotating}
\usepackage{multirow}
\usepackage{color}
\usepackage{url}
\usepackage{subfigure}
\usepackage{rotating}
\usepackage{framed}
\usepackage{mathrsfs}
\usepackage{comment}

\newtheorem{theorem}{Theorem}[section]

\usepackage{lscape}

\newtheorem{definition}{Definition}[section]

\newtheorem{lemma}{Lemma}[section]

\newtheorem{proposition}{Proposition}[section]
\newtheorem{remark}{Remark}[section]

\renewcommand{\S}{\mathsf{S}}

\newcommand{\lam}{\lambda}
\newcommand{\blue}{\textcolor{blue}}
\newcommand{\R}{\mathbb{R}}
\newcommand{\pa}{\partial}
\newcommand{\na}{\nabla}
\newcommand{\eps}{\varepsilon}
\newcommand{\uu}{\mathbf{u}}
\renewcommand{\L}{\mathscr{L}}
\newcommand{\Aop}{\mathscr{A}}
\newcommand{\Lop}{\mathscr{L}}
\newcommand{\Rop}{\mathscr{R}}
\newcommand{\intO}{\int_{\Omega}}
\newcommand{\vp}{\varphi}
\newcommand{\vt}{\vartheta}
\newcommand{\intQT}{\int_0^T\int_{\Omega}}
\newcommand{\LO}[1]{L^{#1}(\Omega)}
\newcommand{\LQ}[1]{L^{\infty}(Q_T)}
\newcommand{\bra}[1]{\left(#1\right)}
\renewcommand{\P}{\mathbb{P}}
\newcommand{\De}{D^\eps}
\newcommand{\lame}{\lambda^{\eps}}
\newcommand{\ue}{u^{\eps}}
\newcommand{\E}{\mathbb{E}}
\newcommand{\sbra}[1]{\left[#1\right]}
\newcommand{\red}{\textcolor{red}}
\newcommand{\bao}[1]{{\color{blue}#1}}
\newcommand{\wh}{\widehat}
\newcommand{\RR}{\mathbf{R}}

\usepackage{authblk}

\author[1]{Malcolm Egan}
\author[2]{Bao Quoc Tang}
\affil[1]{\small Univ Lyon, INSA Lyon, INRIA, Lyon, France\\
{\it email} \href{mailto:malcom.egan@inria.fr}{malcom.egan@inria.fr}}
\affil[2]{\small University of Graz, Graz, Austria\\ {\it email} \href{mailto:quoc.tang@uni-graz.at}{quoc.tang@uni-graz.at}}
\begin{document}
	
\title{Macroscopic limit for stochastic chemical reactions involving diffusion and spatial heterogeneity}
\date{}

\maketitle

\begin{abstract}
    To model bio-chemical reaction systems with diffusion one can either use stochastic, microscopic reaction-diffusion master equations or deterministic, macroscopic reaction-diffusion system. The connection between these two models is not only theoretically important but also plays an essential role in applications. This paper considers the macroscopic limits of the chemical reaction-diffusion master equation for first-order chemical reaction systems in highly heterogeneous environments. More precisely, the diffusion coefficients as well as the reaction rates are spatially inhomogeneous and the reaction rates may also be discontinuous. By carefully discretizing these heterogeneities within a reaction-diffusion master equation model, we show that in the limit we recover the macroscopic reaction-diffusion system with inhomogeneous diffusion and reaction rates.
\end{abstract}

\tableofcontents

\section{Introduction}
Modeling chemical reactions and diffusion have been investigated extensively in the literature due to the frequent appearance of these processes in nature. There are different levels of these models. At the microscopic scale, one can use the stochastic reaction-diffusion master equations (RDME), while at the macroscopic, the reaction-diffusion systems (RDS) are more common. The RDME is useful for simulation as well as when molecular stochasticity due to small quantities is taken into account, while the latter is more convenient to investigate qualitative behaviour of the system. The connection between these models have been investigated since the seventies, with the pioneering works of Kurtz \cite{kurtz1970solutions,Kurtz1971limit} where this issue was studied for the homogeneous (or well-stirred) case, i.e. only chemical reactions are taken into consideration. In these two works, the Markov jump process, which describes chemical reactions in the microscopic level, is shown to converge in probability to the corresponding differential system, which models the same reactions at the macroscopic level. When a diffusion process is also present, attempts to connect microscopic RDME and  macroscopic RDS have been made in \cite{brenig1980stochastic,nicolis1977self}, and a first rigorous proof was provided in \cite{arnold1980deterministic}. In \cite{arnold1980deterministic}, the authors studied the reaction and diffusion processes of {\it a single chemical}, and proved the convergence in the sense of probability of the corresponding Markov jump process to the reaction-diffusion equation. This result was later improved by \cite{kotelenez1986law,kotelenez1986linear,kotelenez1988high,blount1991comparison,blount1993limit} where, among other things, central limit theorems were established for local diffusion, and recently by \cite{watanabe2022comparison} where the diffusion is assumed to be non-local.

\medskip
We highlight that most of existing works have dealt with scalar reaction-diffusion equations. An exception is the recent work \cite{desvillettes2022exponential} where the authors showed exponential convergence to equilibrium for systems with with degenerate reaction rataes; i.e. reactions only occur in a subdomain of positive measure. Our current work extends the literature by considering general first-order, or unimolecular, reaction systems with high levels of heterogeneity, for which we show convergence in probability of microscopic RDME to corresponding RDS. More precisely, we consider chemical reaction networks with an arbitrary number of chemicals, the diffusion and reaction coefficients are spatially inhomogeneous, and in particular, the reaction coefficients can be discontinuous.

\medskip
Another motivation of this paper stems from applications in molecular communication, an emerging field spanning telecommunication, bio- and nano-technologies. The core idea of molecular communication is to send and receive messages using chemical signals mimicking how cells communicate, which was first proposed in \cite{nakano2005molecular, nakano2013molecular}. Understanding dynamics of bio-chemical systems is essential for designing molecular communication schemes. In our recent work \cite{akdeniz2020equilibrium}, we proposed a novel detection scheme called ``equilibrium signaling'' which is robust to the geometry and heterogeneity of considered system. In that work, the microscopic RDME was used for computational purpose while the macroscopic PDE provided quantitative dynamics of the bio-chemical system, and the connection between these models has been assumed therein. Our current work provides rigorous convergence from RDME to PDE for general first-order reaction systems, of which the problem in \cite{akdeniz2020equilibrium} is a special case. 

\medskip
Finally, we remark that the study of RDME for higher order reactions is sparse in the literature. One particular reason for this is that, in the rescaling limit of RDME the biomolecular reactions are lost in two or more dimensions, see \cite{isaacson2009reaction}. Attempts to circumvent this problem includes the convergent RDME, see \cite{isaacson2013convergent}. A general approach has been proposed in \cite{montefusco2023route} where the gradient flow structure of detailed balanced reaction systems has been investigated in a formal setting.
Let us also mention related works considering the limit from system of interacting particles to reaction-diffusion equations \cite{lim2020quantitative} or integral-differential equations \cite{isaacson2022mean}.

\medskip
{\bf Our paper is organized as follows:} in the next section, in order to clearly present ideas and methods, we start with a reversible system of two chemicals in heterogeneous medium in one dimension. For this system, we first consider in Subsection \ref{subsec:1} its macroscopic PDE description and basic properties such as global existence, boundedness and large time behaviour of solutions. The corresponding microscopic RDMS equation is written down in Subsection \ref{subsec:2}. The main result concerning the limit from RDME to PDE model is stated in Subsection \ref{subsec:3} together with its proof. These results are then extended to general first-order reaction networks in Section \ref{sec:general_system} in arbitrary dimension.

\medskip
{\bf Notation}: In the this paper, we will use the following notation:
\begin{itemize}
    \item For $1\le p\le \infty$, $L^p(\Omega)$ denotes the usual Lebesgue space with the associated norm
    \begin{equation*}
        \|u\|_{L^p(\Omega)} = \bra{\int_{\Omega}|u(x)|^pdx}^{1/p} \quad \text{ if } p<\infty,
    \end{equation*}
    and
    \begin{equation*}
        \|u\|_{L^\infty(\Omega)} = \underset{x\in\Omega}{\text{ess}\sup}|u(x)|, \quad \text{ otherwise}.
    \end{equation*}
    \item The Sobolev space $H_0^1(\Omega)$ is defined as
    \begin{equation*}
        H_0^1(\Omega) = \left\{u \in L^2(\Omega)\,:\, \na u\in L^2(\Omega) \text{ and } u|_{\pa\Omega} = 0\right\}
    \end{equation*}
    where the gradient is understood in the distributional sense and the value on the boundary is understood in the trace sense.
    \item For a function $u: \Omega \to \R$, $\Omega\subset\R^n$, the gradient of is defined by
    \begin{equation*}
        \nabla u = \bra{\frac{\pa u}{\pa x_1}, \cdots, \frac{\pa u}{\pa x_n}}.
    \end{equation*}
    For a vector field $F = (F_1, \ldots, F_n): \Omega \to \R^n$, the divergence of $F$ is defined by
    \begin{equation*}
        \na \cdot F = \sum_{i=1}^{n}\frac{\pa F_i}{\pa x_i}.
    \end{equation*}
    \item Consider the random variable $X$ defined on the measurable space $(\mathcal{X},\mathcal{B})$ with underlying probability space $(\Omega_X,\mathcal{F},\mathbb{P})$. The probability distribution of $X$ is denoted by $P$ and the probability of an event $\{X \in B\},~B \in \mathcal{B}$ is denoted by $\mathbb{P}(X \in B)$.  
\end{itemize}

\section{The case of two chemicals}\label{sec:two_chemicals}

\subsection{Macroscopic PDE Model}\label{subsec:1}

Consider the following reversible reaction of two chemical species $\S_1$ and $\S_2$
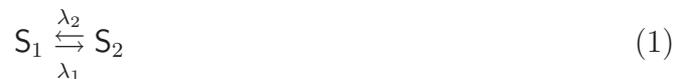
\begin{equation}\label{reaction}
    \S_1 \underset{\lam_1}{\overset{\lam_2}{\leftrightarrows}} \S_2
\end{equation}
where $\lam_1$ and $\lam_2$ are reaction rates. The reactions take place in a bounded vessel $\Omega\subset\R^n$, with Lipschitz boundary $\pa\Omega$. In general, the reaction rates $\lam_1$ and $\lam_2$ are spatially inhomogeneous; that is, the reaction rate depends on the spatial position in the medium with $\lam_i:\Omega \to [0,\infty)$. 

Denote by $\uu_i(x,t)$ the concentration of $\S_i$ at time $t>0$ and position $x\in\Omega$. The chemicals species diffuse through the vessel with the diffusive fluxes 
\begin{equation}\label{diff-flux}
    J_i(x,t) = D_i(x)\nabla \uu_i(x,t), \quad i=1,2.
\end{equation}

By applying the law of mass action, with the diffusive flux given in \eqref{diff-flux}, we obtain the following \textit{macroscopic reaction-diffusion system}
\begin{equation}\label{RS-system}
    \begin{cases}
        \pa_t \uu_1 - \na\cdot(D_1(x)\na \uu_1) = -\lam_1(x)\uu_1 + \lam_2(x)\uu_2, &x\in\Omega, t>0,\\
        \pa_t \uu_2 - \na\cdot(D_2(x)\na \uu_2) = +\lam_1(x)\uu_1 - \lam_2(x)\uu_2, &x\in\Omega, t>0,\\
        \uu_1(x,t) = \uu_2(x,t) = 0, &x\in\pa\Omega, t>0,\\
        \uu_i(x,0) = \uu_{i,0}(x), &x\in\Omega, \; i=1,2.
    \end{cases}
\end{equation}
Here, $\nabla\cdot F$ is the divergence of the vector field $F$.


\begin{lemma}\label{lem:Poincare}
    Assume that $\Omega\subset\R^n$ be a bounded domain with Lipschitz boundary $\pa\Omega$. Then there exists a constant $C_P>0$ such that
    \begin{equation*}
        \|\nabla u\|_{L^2(\Omega)}^2 \ge C_P\|u\|_{L^2(\Omega)}^2 \quad \text{ for all } \quad u\in H_0^1(\Omega).
    \end{equation*}
\end{lemma}

\begin{theorem}[Global bounded weak solution of \eqref{RS-system}]\label{thm:weak_sol}
    Assume the following for each $i\in \{1,2\}$
    \begin{enumerate}[ref=D,label=(D)]
    \item\label{D} $D_i: \Omega \to \R$ satisfies $0< D_* \le D_i(x) \le D^* < +\infty$ for a.e. $x\in\Omega$, and
    \end{enumerate}
    \begin{enumerate}[ref=R,label=(R)]
        \item\label{R}  $\lambda_i \in L^{\infty}_+(\Omega)$, i.e. $\lambda_i\in L^{\infty}(\Omega)$ and $\lambda_i \ge 0$ a.e. in $\Omega$.
    \end{enumerate}
    Then, for any non-negative initial data $\uu_0 = (\uu_{i,0})\in L^{\infty}_{+}(\Omega)^2$, there exists a unique global non-negative, weak solution to \eqref{RS-system} (in the sense of \eqref{def_weak_sol}) which satisfies the mass dissipation
    \begin{equation}\label{mass_dissipation}
        \int_{\Omega}(\uu_1(t)+\uu_2(t))dx \le \int_{\Omega}(\uu_{1,0} + \uu_{2,0})dx, \; \forall t>0.
    \end{equation}
    Moreover, the solution is bounded uniformly in time, i.e. there is $\rho>0$ depending on initial data, $D_*$, $D^*$, $\|\lambda_i\|_{L^\infty(\Omega)}$ such that
    \begin{equation}\label{uniform_bound}
        \sup_{t\ge 0}(\|\uu_1(t)\|_{L^{\infty}(\Omega)} + \|\uu_2(t)\|_{L^{\infty}(\Omega)}) \le \overline{\rho}.
    \end{equation}
    
    In addition, if $D^1, D^2\in C^1(\Omega)\cap C(\overline{\Omega})$, the solution to \eqref{RS-system} can be represented by
    \begin{equation}\label{mild_solution}
        \uu(t) = T(t)\uu_0 + \int_{0}^tT(t-s)\mathbf{R}(\uu(s))ds
    \end{equation}
    where $$\uu_0 = (\uu_{1,0}, \uu_{2,0}), \quad \mathbf{R}(\uu) = \begin{pmatrix}
        -\lam_1(x)\uu_1 + \lam_2(x)\uu_2\\
        \lam_1(x)\uu_1 - \lam_2(x)\uu_2
    \end{pmatrix}$$ and the semigroup $\{T(t) = e^{\L t}\}_{t\ge 0}$ is generated by the operator
    \begin{equation}\label{def:L}
        \L \uu = 
        \begin{pmatrix}
            -\na\cdot(D_1(x)\uu_1) & 0\\
            0 & -\na\cdot(D_2(x)\uu_2)
        \end{pmatrix}
    \end{equation}
    
    Furthermore, if $\lambda_1(x) = \delta \lambda_2(x)$, $x\in \Omega$, for some positive constant $\delta$, then the solution to \eqref{RS-system} decays exponentially to zero, i.e. for any $p \in [2,\infty)$ there are $\alpha_p, C_p>0$ such that 
    \begin{equation}\label{exponential_decay}
        \|\uu_1(t)\|_{L^{p}(\Omega)} + \|\uu_2(t)\|_{L^{p}(\Omega)} \le C_pe^{-\alpha_p t}, \quad \forall t\ge 0.
    \end{equation}
\end{theorem}

\begin{proof}
    A weak solution to \eqref{RS-system} on $(0,T)$, $T>0$, is a pair of functions
    \begin{equation*}
        (\uu_1, \uu_2) \in C([0,T];L^2(\Omega))^{2} \cap L^2(0,T;H^1_0(\Omega))^2
    \end{equation*}
    with
    \begin{equation*}
        \pa_t \uu_i \in L^2(0,T;H^{-1}(\Omega)), i=1,2,
    \end{equation*}
    with $H^{-1}(\Omega)$ being the dual space of $H_0^1(\Omega)$,
    such that for each $i\in \{1,2\}$
    \begin{equation}\label{def_weak_sol}
        \begin{aligned}
        \int_{0}^T\langle \pa_t\uu_i, \varphi\rangle_{H^{-1}(\Omega),H_0^1(\Omega)} dt + \int_0^T\int_{\Omega}D_i(x)\nabla \uu_i \nabla \varphi dxdt\\
        = \int_{0}^T\int_{\Omega}(-\lambda_i \uu_i + \lambda_{3-i}\uu_{3-i})\varphi dxdt
        \end{aligned}
    \end{equation}
    for all test functions $\varphi \in L^2(0,T;H^1(\Omega))$.

    \medskip
    It is noted that the sum of the right-hand side of the equations in \eqref{RS-system} equals to zero.
    Therefore, the global existence and uniform-in-time boundedness of a unique non-negative, weak solution to \eqref{RS-system} follows immediately from \cite[Theorem 1.1]{fitzgibbon2021reaction}.

    \medskip
    It remains to show the decay of the solution in case $\lambda_1(x) = \delta\lambda_2(x)$. Direct computations give
    \begin{equation*}
        \begin{aligned}
            \frac 12\frac{\mathrm{d}}{\mathrm{d}t}\bra{\|\uu_1\|_{\LO{2}}^2 + \delta\|\uu_2\|_{\LO{2}}^2} + \intO \bra{D^1|\na \uu_1|^2 + D^2|\na \uu_2|^2}dx\\
            = -\intO \delta\lambda_2(x)(\uu_1 - \uu_2)^2dx \le 0.
        \end{aligned}
    \end{equation*}
    Thanks to the homogeneous Dirichlet boundary conditions and the Poincar\'e inequality in Lemma \ref{lem:Poincare}, it follows that there is $\beta>0$ satisfying
    \begin{equation*}
        \frac{\mathrm{d}}{\mathrm{d}t}\bra{\|\uu_1\|_{\LO{2}}^2 + \delta\|\uu_2\|_{\LO{2}}^2} + \beta \bra{\|\uu_1\|_{\LO{2}}^2 + \delta\|\uu_2\|_{\LO{2}}^2} \le 0,
    \end{equation*}
    and consequently
    \begin{equation*}
        \|\uu_1(t)\|_{\LO{2}}^2 + \delta\|\uu_2(t)\|_{\LO{2}}^2 \le e^{-\beta t}\bra{\|\uu_{1,0}\|_{\LO{2}}^2 + \|\uu_{2,0}\|_{\LO{2}}^2}.
    \end{equation*}
    Now for any $2<p<\infty$, we have
    \begin{equation*}
        \|\uu_1(t)\|_{L^p(\Omega)}^p = \int_{\Omega}|\uu_1(t)|^pdx \le \|\uu_1(t)\|_{L^{\infty}(\Omega)}^{p-2}\|\uu_1(t)\|_{L^2(\Omega)}^2 \le C_0e^{-\beta t}
    \end{equation*}
    and thus 
    \begin{equation*}
        \|\uu_1(t)\|_{L^p(\Omega)} \le C_0^{1/p}e^{-(\beta/p)t} \quad \forall t>0.
    \end{equation*}
    The proof for $\uu_2$ follows the same way.

    \medskip
    When the diffusion coefficients $D_1, D_2$ are in $C^1(\Omega)\cap C(\overline{\Omega})$, it is known, see e.g. \cite{zheng2004nonlinear} that the operator $\L$ defined in \eqref{def:L} generates an analytic contraction semigroup $\{T(t) = e^{\L t}\}_{t\ge 0}$, using which the solution to \eqref{RS-system} can be written as \eqref{mild_solution}. Note that this solution is also a weak solution and therefore enjoys the uniform-in-time bounds \eqref{uniform_bound} as well as the exponential decay \eqref{exponential_decay}.
\end{proof}
Since the system is linear, the global existence is not surprising, and one can get a bound in $L^{\infty}(\Omega)$-norm \text{which might depend on time} (typically exponential). The uniform-in-time bound \eqref{uniform_bound}, which is a consequent of the dissipation of mass \eqref{mass_dissipation}, plays an important role in defining the stopping process in the sequel (see \eqref{stopping_process}).

\subsection{Reaction-diffusion master equations}\label{subsec:2}
Let $\Omega = (0,1)$, and consider a volume in one dimension, i.e. $V = (0,L)$ for some $L>0$. The case of a cube $\Omega = (0,1)^n$ and $V = (0,L)^n$ in $\R^n$ will be discussed in Section \ref{sec:general_system} for general systems. At the microscopic level, the system \eqref{reaction} with diffusion can be modeled via a continuous-time Markov jump process \cite{arnold1980deterministic,kotelenez1986law}. Let $N\in \mathbb N$ and consider a uniform partition of $V$ into $N$ cells of equal size $w = L/N$ and define the $j$-th cell by $(x_{j-1},x_{j}] = ((j-1)w, jw]$, $j=1,\ldots, N$. It is remarked that this partition is not a discretization of the domain $\Omega$ in the macroscopic model. In fact, later on, it will be assumed that the volume $V$, as well as each cell, will "blow up", i.e. become unbounded (see \eqref{assumption}). To account for the boundary conditions detailed in the sequel, the volume is extended to an interval of length $L + 2w$ consisting of $N + 2$ cells of length $w$. The number of molecules of each species $l\in \{1,2\}$ in cell $j = 0,\ldots,N+1$ at time $t$ is denoted by $\widetilde  X^l_j(t)$. We assume that the number of molecules for the cells $j = 0$ and $j = N+1$ are always zero, which corresponds to the homogeneous Dirichlet boundary condition in (\ref{RS-system}). The total state of the system is denoted by $\widetilde X(t)$, which forms a continuous-time Markov jump process on $\mathbb{N}^{2(N+2)}$.

\medskip
Suppose that at time $t$, $\widetilde X_j^l(t) = k_{j}^l$. Let $m = (m^1, m^2) \in \mathcal \mathbb{N}^{2(N+2)}$  and denote the $j$-th unit vector in $\mathbb R^{N+2}$ by $e_j$. The transition rates for the process $\widetilde X(t)$ are given by 
\begin{align}\label{eq:trans_rates}
&q_{k,k+m} = \begin{cases}
\lambda_j^l k_j^l,~&m^l = -e_j,~m^{3-l} = e_j,~j = 1,\ldots,N,~l = 1,2,\\
\lambda_j^{3-l} k_j^{3-l},~&m^l = e_j,~m^{3-l} = -e_j,~j = 1,\ldots,N,~l = 1,2,
\end{cases}\\
&q_{k,k+m} = \begin{cases}
D_{j}^l N^2k_j^l,~&m^l = -e_j + e_{j+ 1},~j = 1,\ldots,N,~l = 1,2,\\
D_{j}^l N^2k_j^l,~&m^l = -e_j + e_{j- 1},~j = 1,\ldots,N,~l = 1,2,\\
D_{j+1}^l N^2 k_{j + 1}^l,~&m^l = e_j - e_{j+1},~j = 1,\ldots,N-1,~l = 1,2,\\
D_{j-1}^l N^2 k_{j-1}^l,~&m^l = e_j - e_{j-1},~j = 2,\ldots,N,~l = 1,2.
\end{cases}
\end{align}
and $q_{k,k+m} = 0$ otherwise. Here, $D_j^l$ and $\lambda_j^l$ denote the diffusion coefficient and reaction rate constant in the $j$-th cell for species $l$, which are calculated from the functions $D_l$ and $\lambda_l$ via
\begin{equation}\label{approx_Dlambda}
    D_j^l = {w^{-1}}\int_{x_{j-1}}^{x_j}D_l(x/V)dx, \quad \lambda^l_j = {w^{-1}}\int_{x_{j-1}}^{x_j}\lambda_l(x/V)dx.
\end{equation}
It is remarked that we only need integrability of $D_l$ and $\lambda_l$ for these voxel diffusion and reaction rates to be meaningful. 

\medskip
The evolution of $\widetilde X(t)$ is 
described by the \textit{reaction-diffusion master equation} (RDME); i.e., the Kolmogorov forward equation for the Markov jump process described by the transition rates in (\ref{eq:trans_rates}). Let $k_0 \in \mathbb{N}^{2(N+2)}$, $P(k,t) = \mathrm{Pr}(\widetilde X(t) = k|\widetilde X(0) = k_0),~k \in \mathbb{N}^{2(N+2)}$ and $e_j^l = e_{(l-1)N + j} \in \mathbb{N}^{2(N+2)}$. The RDME is then given by 
\begin{align}
	\frac{\mathrm{d}}{\mathrm{d}t} P(k,t) &= \sum_{l=1}^2 \sum_{j=1}^N \left(q_{k- e_j^l +e_{j+1}^l,k}P(k - e_j^l + e_{j+1}^l) - q_{k,k - e_j^l + e_{j+1}^l}P(k,t)\right.\notag\\
	&~~~ \left. + q_{k - e_j^l + e_{j-1}^l,k}P(k - e_j^l + e_{j-1}^l) - q_{k,k - e_j^l + e_{j-1}^l}P(k,t)\right)\notag\\
	&~~~ + \sum_{l=1}^2 \sum_{j=1}^N q_{k + e^{3-l}_j - e_j^l,k}P(k + e^{3-l}_j - e_j^l) - q_{k,k+e_j^{3-l} - e_j^l}P(k,t).
\end{align}
Recall that, for all $t > 0$, the boundary layers satisfy $\tilde{X}_j^l(t) = 0,~j = 0,N+1$ to enforce the zero boundary condition.


\medskip


We now define the concentration process $C(t)$ via
\begin{equation}
\begin{aligned}
&C_j^l(t) = \frac{\widetilde X_j^l(t)}{w},~j = 1,\ldots,N\\
&C_j^l(t) = 0,~j = 0,N-1\\
\end{aligned}
\end{equation}
where the last two conditions account for the zero boundary condition. 
The process $C(t)$ forms a continuous-time Markov jump process, a property inherited from the process $\widetilde X(t)$. 
Observe that
\begin{align}
q_{k,k+m} = wf(k/w,m),~k \in \mathbb{N}^{2(N+1)},
\end{align}
where $f:\mathbb{R}_+^{2(N+2)} \times \mathbb{N}^{2(N+1)} \rightarrow \mathbb{R}_+$ with
\begin{equation}\label{def_f}
\left\{
\begin{aligned}
&f(c,m) = \begin{cases}
\lambda_j^l c_j^l,~&m^l = -e_j,~m^{3-l} = e_j,~j = 1,\ldots,N,~l = 1,2,\\
\lambda_j^{3-l} c_j^{3-l},~&m^l = e_j,~m^{3-l} = -e_j,~j = 1,\ldots,N,~l = 1,2,
\end{cases}\\
&f(c,m) = \begin{cases}
D_{j}^l N^2c_j^l,~&m^l = -e_j + e_{j+ 1},~j = 1,\ldots,N,~l = 1,2,\\
D_{j}^l N^2c_j^l,~&m^l = -e_j + e_{j- 1},~j = 1,\ldots,N,~l = 1,2,\\
D_{j+1}^l N^2 c_{j + 1}^l,~&m^l = e_j - e_{j+1},~j = 1,\ldots,N,~l = 1,2,\\
D_{j-1}^l N^2 c_{j-1}^l,~&m^l = e_j - e_{j-1},~j = 1,\ldots,N,~l = 1,2.
\end{cases}\\
&f(c,m) = 0 \quad \text{ otherwise.}
\end{aligned}\right.
\end{equation}
It follows that $C(t)$ is a density dependent continuous Markov process \cite{Kurtz1971limit} with waiting time parameter 
\begin{align}\label{tau}
\tau(c) = w\sum_{m\in \mathbb N^{2(N+2)}} f(c,m)
\end{align}
and the jump distribution function given by
\begin{align}\label{sigma}
\sigma(c,m/w) = \frac{f(c,m)}{\sum_{m\in \mathbb N^{2(N+2)}} f(c,m)} = \frac{wf(c,m)}{\tau(c)}.
\end{align}
We remark that due to the definition \eqref{def_f}, the sum over all $m\in \mathbb{N}^{2(N+2)}$ in \eqref{tau} and \eqref{sigma} are indeed finite sums.

\subsection{The macroscopic limit}\label{subsec:3}

\medskip
To allow for a comparison between the microscopic and macroscopic models, we interpret $C(t)$ as a function on $[0,1]$. In particular, let 
\begin{align}
    u_l(x,t) = \sum_{j=1}^N C_j^l(t)\mathbf{1}\left\{x \in \left[\frac{(j-1)w}{V},\frac{jw}{V}\right)\right\},~l = 1,2.
\end{align}
The function $u(t)$ lies in a subspace $\mathcal{X}_N$, which is a subspace of $L_2[0,1]\times L_2[0,1]$ with inner product 
\begin{align}
    \langle u,v\rangle = \frac{w}{V} \sum_{l = 1}^2 \sum_{j=1}^N u_l\left(\frac{jw}{V}\right)v_l\left(\frac{jw}{V}\right) = \sum_{l = 1}^2 \int_0^1 u_l(x)v_l(x)\mathrm{d}x,~u,v \in \mathcal{X}_N.
\end{align}


Define $\tau_S$ as the first exit time of $u(t)$ from 
\begin{align}\label{stopping_process}
	S_{\rho} = \left\{u \in \mathcal{X}_N^+: \sup_{0 \leq x \leq 1} (u_1(x) + u_2(x)) \leq \rho\right\}
\end{align}
with arbitrary $\rho > \overline{\rho}$, where $\overline{\rho}$ is the uniform-in-time bound \eqref{uniform_bound} of solutions in Theorem \ref{thm:weak_sol}. The stopped process $\tilde{u}(t)$ is then defined as
\begin{align}\label{eq:tilde_x}
\tilde{u}(\cdot,t) = u(\cdot,\min\{t,\tau_{S_\rho}\}).
\end{align}
Let 
\begin{align}
    c_l(u) = \left(u_l\left(0\right),\ldots,u_l\left(\frac{(N-1)w}{V}\right)\right).
\end{align}
Then, $\tilde{u}(t)$ is a jump Markov process with
\begin{align}\label{new_stop_process}
\tilde{\tau}(u) = \tau(c(u))\mathbf{1}\left\{u \in \mathcal{S}_{\rho}\right\} \quad \text{ and } \quad
\tilde{\sigma}(u,m/w) = \left\{\begin{array}{lr}
\sigma(c(u),m/w), & u \in S_{\rho}\\
0, & u \not\in S_{\rho},
\end{array}\right.
\end{align}

The main result of the current work is the following theorem. 

\begin{theorem}\label{thrm:conv_smooth}
    Assume that diffusion coefficients $D_i\in C^1(\Omega)\cap C(\overline{\Omega})$, $i=1,2$, reaction rates $\lam_i\in L^{\infty}(\Omega)$, $i=1,2$, satisfy \eqref{D} and \eqref{R}. 
    Let $\tilde{u}(t)$ be as defined in (\ref{eq:tilde_x}) and $\uu(t)$ the bounded weak solution to (\ref{RS-system}) obtained in Theorem \ref{thm:weak_sol}. Suppose that 
\begin{enumerate}[ref=H\theenumi,label=(H\theenumi)]
    \item\label{H1} $\|\tilde{u}(0) - \uu(0)\| = 0$ in probability; and
    \item\label{H2} the following limit holds \begin{equation}\label{assumption}
    \dfrac{N^2}{w} \rightarrow 0 \quad \text{as}\quad  N,w \rightarrow \infty.
    \end{equation}
\end{enumerate}
Then for any $\delta>0$, and any $t>0$
\begin{align}
    \P\bra{\|\tilde{u}(t) - \uu(t)\| \ge \delta} \rightarrow 0 \quad \text{ as } N, w\to\infty.
\end{align}
\end{theorem}

Theorem~\ref{thrm:conv_smooth} generalizes the result of \cite{arnold1980deterministic} to the case of reaction {\it systems} with {\it inhomogeneous} diffusion and reactions.

\medskip
The rest of this section is devoted to prove Theorem \ref{thrm:conv_smooth}. We start off with the behaviour of a martingale related to the space-time Markov jump process $\tilde{u}$. Then discretisation of the inhomogeneous diffusion and its corresponding semigroup are considered, which in turn will be used in comparing $\tilde{u}$ and $\uu$ in the last part. 

\subsubsection{An accompanying martingale}
Consider the operator on $\mathcal{X}_N$
\begin{align}\label{def_F}
F(u) = \sum_m mf(u,m),
\end{align}
which is viewed as a stochastic approximation of the average infinitesimal rate of change of $u$. We now establish some properties of
\begin{align}\label{split}
\tilde{z}(t) = \tilde{u}(t) - \tilde{u}(0) - \int_0^t F(\tilde{u}(s))\mathrm{d}s.
\end{align}


\begin{lemma}
The process \eqref{split}
is a martingale with respect to $P_u$ (the distribution induced by an initial value $u$). 
\end{lemma}

\begin{proof}
With $\tilde{\tau}$ defined in \eqref{new_stop_process} we observe that 
\begin{align}
\sup_u \tilde{\tau}(u) \leq  2Nw\rho\left(\|\lambda_1\|_{L^{\infty}(\Omega)} + \|\lambda_2\|_{L^{\infty}(\Omega)} + 4N^2 D^*\right) < \infty
\end{align}
and 
\begin{align}
\sup \tilde{\tau}(u) \sum_m \frac{|m|}{w} \tilde{\sigma}(u,m/w) &= \sup_{u \in S_{\rho}} \sum_m |m| f(u,m)\notag\\
&\leq 2\sqrt{2}N\rho\left(\|\lambda^1\|_{L^{\infty}(\Omega)} + \|\lambda^2\|_{L^{\infty}(\Omega)} + 4N^2 D^*\right) \notag\\
&< \infty.
\end{align}

By \cite[Proposition 2.1]{Kurtz1971limit}, the process \eqref{split} is then a martingale with respect to $P_u$. 
\end{proof}

Observe that for an arbitrary initial distribution,
\begin{align}
	\mathbb{E}[\tilde z(t)] = 0,~t \geq 0.
\end{align} 

\begin{proof}
	Since $f$ is the infinitesimal generator of the process $\tilde{u}$, the result follows by definition.
\end{proof}

\begin{lemma}\label{lem:scale}
	Under the assumption \eqref{H2}, it holds for all $\delta > 0$,
	\begin{equation*}
	    \mathbb{P}\left(\sup_{0\le s\le t}\|\tilde{z}(s)\| > \delta\right) \rightarrow 0, \quad\text{ as } N,w \rightarrow \infty.
	\end{equation*}
\end{lemma}

\begin{proof}
    Doob's $L^p$-inequality \cite[Theorem II.1.7]{Revuz1999continuous} yields 
\begin{align}
P_u\left(\sup_{0 \leq s \leq t} \|\tilde{z}(s)\| > \delta\right) \leq \delta^{-2} \mathbb{E}_u[\|\tilde{z}(t)\|^2].
\end{align}
An estimate for $\mathbb{E}_u[\|\tilde{z}(t)\|^2]$ is obtained as follows. From Kurtz \cite[Lemma 2.9]{Kurtz1971limit}, 
\begin{align}
\mathbb{E}_u[\|\tilde{z}(t)\|^2] &= \int_0^t \mathbb{E}_u\left[\tilde{\tau}(\tilde{u}(s))\left(\sum_m \|m/w\|^2 \tilde{\sigma}(\tilde{u}(s),m/w)\right)\mathrm{d}s\right]\notag\\
&= \int_0^t \mathbb{E}_u\left[\sum_m \|m/w\|^2 w f(\tilde{u}(s),m)\right]\mathrm{d}s\notag\\
&= (wN)^{-1} \int_0^t \mathbb{E}_u\left[\sum_m |m|^2 f(\tilde{u}(s),m)\right]\mathrm{d}s\notag\\
&\leq w^{-1} \int_{0}^t \sum_{l = 1}^2 \mathbb{E}_u\left[\int_{0}^1 (\lambda_{3-l} \tilde{u}(x,s) + \lambda_l \tilde{u}(x,s))\mathrm{d}x\right]\mathrm{d}s\notag\\
&~~~ + w^{-1} \sum_{l=1}^2\int_{ 0}^t \mathbb{E}_u\left[\int_{0}^1 4N^2 D^* \tilde{u}(x,s)\mathrm{d}x\right]\mathrm{d}s\notag\\
&\leq 2tw^{-1}\rho(\|\lambda_1\|_{\LO{\infty}} + \|\lambda_2\|_{\LO{\infty}} + 4N^2 D^*)\label{martingale_inequality}
\end{align}
noting that $\mathbb{E}_u[\|\tilde{z}(t)\|^2] = 0,~u \not\in S_{\rho}$. Thanks to the assumption \eqref{H2}, we have $\|\tilde{z}(t)\|\to 0$ as $w,N\to\infty$ in mean, and consequently in probability.
\end{proof}



\subsubsection{Discretization of diffusion and reaction terms}

Due to the inhomogeneity of diffusion coefficients and reaction rate constants, we approximate them by the following piecewise constants functions: for $l=1,2$,
\begin{equation*}
    \wh{D}^l(r) = D_j^l \quad\text{ and } \quad \wh{\lambda}^l(x) = \lambda^l_j,\quad \text{for}\quad x\in I_j = [x_{j-1},x_j),
\end{equation*}
where $D^l_j$ and $\lambda^l_j$ are defined as in \eqref{approx_Dlambda}. Now to study the corresponding discrete evolution we define the discrete reactions
\begin{equation}\label{discrete_reaction}
    \RR_N^l(u)(x) = -\wh{\lambda}^l(x)u^l(x) + \wh{\lambda}^{3-l}(x)u^{3-l}(x)
\end{equation}
and the discrete diffusion operator
\begin{equation}\label{discrete_diffusion}
    \L_N^l(u)(x) = h^{-2}\pa_{h}(\wh{D}^l(x)\pa_{-h}u^l(x))
\end{equation}
where the difference quotients are defined as
\begin{equation*}
    \pa_{h}f(x) = h^{-1}(f(x+h)-f(x)) \quad \text{ and } \quad \pa_{-h}f(x) = h^{-1}(f(x)-f(x-h)).
\end{equation*}
Using the definition of $\wh{D}^l$, we can rewrite $\L_N^l$ as
\begin{equation}\label{def:LN}
    \begin{gathered}
    \L_N^l(u)(x) = h^{-2}\sbra{D_{j+1}^lu^l(x+h) - (D_{j+1}^l+D_j^l)u^l(x) + D_j^lu^l(x-h)},\\
    \L_N = (\L_N^l)_{l=1,2}
    \end{gathered}
\end{equation}
for $x \in I_j$, $1\le j \le N-1$. Note that from the definition of $f$ in \eqref{def_f}, $F$ in \eqref{def_F}, we have
\begin{equation}\label{important_relation}
    F(u) = \L_N(u) + \RR_N(u).
\end{equation}

\begin{remark}
    The form of $\L_N^l(u)(x)$ in (\ref{discrete_diffusion}) generalizes the analogous expression for spatially homogeneous diffusion, as considered in, e.g., \cite{arnold1980deterministic,kotelenez1986law}. It is not the only potential generalization; however, as we will see, is the correct choice to establish the desired convergence result. 
\end{remark}

\begin{lemma}
	The operator $\L_N^l$, $l = 1,2$, defined in \eqref{discrete_diffusion} is a self-adjoint operator in $L^2(\Omega)$.
\end{lemma}
\begin{proof}
    For simplicity, we omit the superscript $l$ in $\L_N^l$. We show that
    \begin{equation}\label{e1}
        \int_{0}^1(\L_N u)(x) v(x) dx = \int_{0}^1u(x)(\L_N v(x))dx.
    \end{equation}
    For $u, v\in H_0^1(\Omega)$, we have
    \begin{align*}
        \int_0^1\pa_h u(x)v(x)dx &= \sum_{j=0}^{N-1}\int_{x_j}^{x_{j+1}}h^{-1}(u(x+h) - u(x))v(x)dx\\
        &= h^{-1}\sum_{j=1}^{N}\int_{x_{j}}^{x_{j+1}}u(x)v(x-h)dx - h^{-1}\sum_{j=0}^{N-1}\int_{x_j}^{x_{j+1}}u(x)v(x)dx\\
        &= h^{-1}\sum_{j=0}^{N-1}\int_{x_j}^{x_{j+1}}u(x)(v(x-h) - v(x))dx\\ &\quad + h^{-1}\int_{x_{N}}^{x_{N+1}}u(x)v(x-h)dx - h^{-1}\int_{x_0}^{x_1}u(x)v(x-h)dx\\
        &= -\int_{0}^1u(x)\pa_{-h}v(x)dx
    \end{align*}
    where we used $v|_{[-h, 0]} = u|_{[1, 1+h]} = 0$ at the last step. Now we apply this ``integration by parts'' to obtain
    \begin{align*}
        \int_{0}^{1}(\L_N u(x))v(x)dx &= \int_0^1\pa_h(\wh{D}(x)\pa_{-h}u(x)) v(x)dx\\
        &= -\int_0^1\wh{D}(x)\pa_{-h}u(x)\pa_{-h}v(x)dx\\
        &= \int_0^1u(x)\pa_h(\wh{D}(x)\pa_{-h}v(x))dx = \int_0^1u(x)\L_N v(x)dx,
    \end{align*}
    which is the self-adjointness of the operator $\L_N$.
\end{proof}

\begin{remark}
    We note that, unlike the case of spatially homogeneous diffusion coefficients as considered in \cite{arnold1980deterministic,kotelenez1986law}, the operator $\mathcal{L}_N$ is not self-adjoint under Neumann boundary conditions (i.e., a reflective boundary). 
\end{remark}

\begin{lemma}\label{lem:semigroup}
The semigroup $T_N(t)$ defined as \begin{align*}
    T_N(t):= \exp(-\L_N t),~t \geq 0.
\end{align*}
is an analytic contraction semigroup in $L^2(\Omega)$ satisfying
\begin{enumerate}
    \item[(i)] $\|T_N(t)\| \leq e^{\omega t}$ for some $\omega \geq 0$.
    \item[(ii)] $\|T_N(t)\L_N\| \leq \dfrac{1}{t},~t > 0$.
\end{enumerate}
\end{lemma}
\begin{proof}
    We will apply the Hille-Yosida Theorem. Due to the definition, we have $D(-\L_N) = \{ u\in C^0(\Omega):\, u|_{\partial\Omega} = 0\}$ which is dense in $L^2(\Omega)$. Let $\omega >0$. We claim that for each $v\in L^2(\Omega)$, there exists a unique solution to the equation
    \begin{equation}\label{f1}
        \omega u - \L_Nu = v.
    \end{equation}
    Indeed, by defining the bilinear form $a$
    \begin{equation*}
        a(u,z) = \omega \langle u, z\rangle + \langle -\L_N u, z\rangle 
    \end{equation*}
    we have
    \begin{align*}
        |a(u,z)| &\le \omega\|u\|\|z\| +\int_0^1D(x)|\partial_{-w}u(x)||\partial_{-w}z(x)|dx\\
        &\le \omega \|u\|\|z\| + w\|D\|_{L^{\infty}(\Omega)}\int_0^1|u(x+w) - u(x)||z(x+w) - z(x)|dx\\
        &\le (\omega + 4w\|D\|_{L^{\infty}(\Omega)})\|u\|\|z\|
    \end{align*}
    and
    \begin{align*}
        a(u,u) &= \omega\|u\|^2 + \langle -\L_N u, u \rangle\\
        &= \omega\|u\|^2 + \int_0^1D(x)|\partial_{-w}u(x)|^2dx \geq \omega\|u\|^2.
    \end{align*}
    Therefore, thanks to Lax-Milgram theorem, \eqref{f1} has a unique solution $u\in L^2(\Omega)$, which means that $(\omega I - \L_N)^{-1}$ is one-to-one and onto. It remains to show that 
    \begin{equation}\label{f1_0}
        \|(\omega I - \L_N)^{-1}\| \le \frac{1}{\omega}.
    \end{equation}
    By multiplying \eqref{f1} by $u$ in $L^2(\Omega)$ and using $\langle -\L_Nu , u\rangle \ge 0$ we have
    \begin{align*}
        \omega\|u\|^2 = \langle v, u \rangle \le \frac{\omega}{2}\|u\|^2 + \frac{1}{2\omega}\|v\|^2
    \end{align*}
    which implies
    \begin{equation}\label{f2}
        \|u\|^2 \le \frac{1}{\omega^2}\|v\|^2.
    \end{equation}
    Thus
    \begin{align*}
        \|(\omega I - \L_N)^{-1}\| &= \sup_{\|v\| = 1}\|(\omega I - \L_N)^{-1}v\|\\
        &= \sup_{\|v\| = 1}\{\|u\|:\, v = (\omega I - \L_N)u\}\\
        &\le \sup_{\|v\| = 1}\frac{1}{\omega}\|v\| \le \frac{1}{\omega}\qquad (\text{thanks to } \eqref{f2})
    \end{align*}
    which shows \eqref{f1_0}. Therefore, we can apply Hille-Yosida theorem to obtain that $T_N(t) = \exp(-\L_N t)$ is an analytic contraction semigroup in $L^2(\Omega)$. The properties (i) and (ii) therefore follow immediately from contraction semigroup.
\end{proof}

\subsubsection{Proof of Theorem \ref{thrm:conv_smooth}}
Using \eqref{mild_solution}, \eqref{split},  and \eqref{important_relation}, we can split the difference $\tilde u - \uu$ as
\begin{equation}\label{h1}
\begin{aligned}
    \tilde u(t) - \uu(t) = \tilde z(t) + \tilde u(0) &+ \int_0^tF(\tilde{u}(s))ds - T(t)\uu(0) - \int_0^tT(t-s)\RR(\uu(s))ds\\
    = \tilde z(t) + \tilde u(0) &+ \int_0^t\L_N(\tilde{u}(s))ds+ \int_0^t\RR_N(\tilde{u}(s))ds\\
    & - T(t)\uu(0) - \int_0^tT(t-s)\RR(\uu(s))ds.
\end{aligned}
\end{equation}
Let $y(t)$ be the solution to
\begin{equation}\label{def_y}
    y(t) = \tilde{u}(0) + \int_0^t\L_Ny(s)ds + \int_0^t\RR_N(\tilde{u}(s))ds,
\end{equation}
which is equivalent to
\begin{equation*}
    y(t) = T_N(t)\tilde{u}(0) + \int_0^tT_N(t-s)\RR_N(\tilde{u}(s))ds.
\end{equation*}
Therefore, we have
\begin{align*}
    \int_0^t\RR_N(\tilde{u}(s))ds &= y(t) - \tilde{u}(0) - \int_0^t\L_Ny(s)ds\\
    &= T_N(t)\tilde{u}(0) + \int_0^tT_N(t-s)\RR_N(\tilde{u}(s))ds- \tilde{u}(0) - \int_0^t\L_Ny(s)ds.
\end{align*}
Inserting this into \eqref{h1} and taking the norm of both sides give
\begin{equation}\label{h2}
\begin{aligned}
    \|\tilde{u}(t) - \uu(t)\| \le \|\tilde z(t)\| &+ \underbrace{\left\|\int_0^t\L_N(\tilde u(s) - y(s))ds\right\|}_{(I)} + \underbrace{\|T_N(t)\tilde{u}(0) - T(t)\uu(0)\|}_{(II)} \\
    &+ \underbrace{\left\|\int_0^t\sbra{T_N(t-s)\RR_N(\tilde{u}(s)) - T(t-s)\RR(\uu(s))}ds\right\|}_{(III)}. 
\end{aligned}
\end{equation}
We estimate the terms $(I), (II)$ and $(III)$ separately.

\medskip
\noindent{\underline{Estimate of $(II)$.}}
\begin{lemma}\label{lem:semigroup_convergence}
    Let $\{T_N(t)\}_{t\ge 0}$ and $\{T(t)\}_{t\ge 0}$ be the semigroups generated by $\L_N$ and $\L$ respectively. We have
    \begin{align}\label{est_II}
        \|T_N(t)\tilde{u}(0) - T(t)\uu(0)\| \leq \epsilon_1(t) + \|\tilde{u}(0) - \uu(0)\|,
    \end{align}
    where 
    \begin{equation}\label{eps1}
    \epsilon_1(t) \rightarrow 0 \quad\text{ as } N\to \infty
    \end{equation}
    for each $t>0$.
\end{lemma}

\begin{proof}
    By triangle inequality, it holds
    \begin{align*}
        \|T_N(t) \tilde{u}(0) - T(t)\uu(0)\| &\le \|T_N(t)\tilde{u}(0) - T_N(t)\uu(0)\| + \|(T_N(t)-T(t))\uu(0)\|\\
        &\le \|\tilde{u}(0) - \uu(0)\| + \|(T_N(t)-T(t))\uu(0)\|
    \end{align*}
    where we used that $\{T_N(t)\}_{t\ge 0}$ is a contraction semigroup at the last step. It remains to show
    \begin{equation*}
        \epsilon_1(t):= \|(T_N(t) - T(t))\uu(0)\| \to 0
    \end{equation*}
    as $N\to\infty$ for each $t>0$.     This is nothing else but the $L^2(\Omega)$-convergence of finite difference scheme for the diffusion equation $\pa_t u - \L u = 0$ where  coefficients of $\L$ are smooth, in this case, $C^1$, and therefore the convergence follows from e.g. \cite{jovanovic1989convergence}.
\end{proof}

\medskip
\noindent{\underline{Estimate of $(I)$.}}
We show that 
\begin{equation}\label{f3}
    \int_0^t\L_N(\tilde{u}(s)-y(s))ds= \int_0^tT_N(t-s)\L_N\tilde{z}(s)ds
\end{equation}
Indeed, we compute using \eqref{split} and \eqref{def_y}
\begin{align*}
    w(t) &=\int_0^t\L_N(\tilde{u}(s) - y(s))ds\\
    &= \int_0^t\bra{\L_N\tilde{z}(s) + \L_N\int_0^s\L_N(\tilde{u}(\tau) - y(\tau))d\tau}ds\\
    &= \int_0^t\L_N\tilde{z}(s)ds + \int_0^t\L_Nw(s)ds.
\end{align*}
It follows that 
\begin{equation*}
    w(t) = T_N(t)w(0) + \int_0^tT_N(t-s)\L_N\tilde{z}(s)ds
\end{equation*}
which is \eqref{f3} due to $w(0) = 0$. Further calculations lead to
\begin{align*}
    \int_0^t\L_N(\tilde{u}(s)-y(s))ds &= \int_0^tT_N(t-s)\L_N(\tilde{z}(s)-\tilde{z}(t))ds + \int_0^tT_N(t-s)\L_N\tilde{z}(t)ds\\
    &= \int_0^tT_N(t-s)\L_N(\tilde{z}(s)-\tilde{z}(t))ds + (T_N(t)-\text{Id})\tilde{z}(t).
\end{align*}
Therefore
\begin{equation}\label{est_I}
\begin{aligned}
    (I) = \left\|\int_0^t\L_N(\tilde{u}(s)-y(s))ds\right\| &\le \int_0^t\|T_N(t-s)\L_N(\tilde{z}(s)-\tilde{z}(t))\|ds + 2\|\tilde{z}(t)\|\\
    &\le \int_0^t\frac{\|\tilde{z}(t)-\tilde{z}(s)\|}{t-s}ds + 2\|\tilde{z}(t)\|
\end{aligned}
\end{equation}
where we used Lemma \ref{lem:semigroup} (ii) at the last estimate.

\medskip
\noindent{\underline{Estimate of $(III)$.}}
Now consider 
\begin{align*}
    (III)  &=\left\|\int_0^t [T_N(t-s)\RR_N(\tilde{u}(s)) -  T(t-s)\RR(\uu(s))]\mathrm{d}s\right\|\\
    &\le \left\|\int_0^t T_N(t-s)[\RR_N(\tilde{u}(s)) - \RR_N(\uu(s)))]\mathrm{d}s\right\| + \left\|\int_0^t T_N(t-s)[\RR_N\uu(s) - \RR(\uu(s))]\mathrm{d}s\right\|\\
    &\quad + \left\|\int_0^t[T_N(t-s)-T(t-s)]\RR(\uu(s))ds\right\|\\
    &=: (III_1) + (III_2) + (III_3).
\end{align*}
Using the contraction property of the semigrioup $T_N(t)$ and boundedness of the reaction rate constants $\lambda^1(\cdot)$ and $\lambda^2(\cdot)$, we have
\begin{equation*}
    (III_1) \le c_0\int_0^t\|\tilde{u}(s) - \uu(s)\|ds
\end{equation*}
where the constant $c_0$ depends on $\|\lam^1\|_{\LO{\infty}}$ and $\|\lam^2\|_{\LO{\infty}}$. For $(III_2)$, we use the contraction of $\{T_N(t)\}$ again to have
\begin{equation*}
    (III_2) \le \int_0^t\|\RR_N(\uu(s)) - \RR(\uu(s))\|ds =: \epsilon_2(t)
\end{equation*}
with the property
\begin{equation}\label{eps2}
    \lim_{N\to \infty}\epsilon_2(t) = 0
\end{equation}
for each $t>0$, thanks to the definition of $\RR_N$ in \eqref{discrete_reaction}. Finally, by using similar arguments to the proof of Lemma \ref{lem:semigroup_convergence}, we have
\begin{equation}\label{eps3}
    \epsilon_3(t):= (III_3) = \left\|\int_0^t[T_N(t-s)-T(t-s)]\RR(\uu(s))ds\right\| \to 0 \quad \text{ as } N\to \infty,
\end{equation}
for each $t>0$. These estimates yield
\begin{equation}\label{est_III}
    (III) \le c_0\int_0^t\|\tilde{u}(s) - \uu(s)\|ds + \epsilon_2(t) + \epsilon_3(t).
\end{equation}
\medskip
\begin{proof}[Proof of Theorem \ref{thrm:conv_smooth}]
    Combining \eqref{h2}, \eqref{est_II}, \eqref{est_I} and \eqref{est_III}, we obtain
    \begin{equation}\label{f5}
        \|\tilde{u}(t) - \uu(t)\| \le 3\|\tilde{z}(t)\| + \int_0^t\frac{\|\tilde{z}(t)-\tilde{z}(s)\|}{t-s}ds + \sum_{i=1}^{3}\epsilon_i(t)+ c_0\int_0^t\|\tilde{u}(s) - \uu(s)\|ds.
    \end{equation}
    We now show that
    \begin{equation*}
        \int_0^t\frac{\|\tilde{z}(t)-\tilde{z}(s)\|}{t-s}ds \to 0 \quad \text{as } N\to\infty
    \end{equation*}
    in probability. Indeed, from the fact that $\tilde{z}(t)$ is a martingale, it follows that $\tilde{z}(t) - \tilde{z}(s)$ is also a martingale. We therefore can apply estimate \eqref{martingale_inequality} to get 
    \begin{equation*}
        \mathbb{E}[\|\tilde{z}(t)-\tilde{z}(s)\|^2] \le 2(t-s)w^{-1}\rho\left(\|\lam^1\|_{\LO{\infty}}+\|\lam^1\|_{\LO{\infty}} + 4N^2 \right) =: 2A(w,N)(t-s)
    \end{equation*}
    with
    \begin{equation*}
        \lim_{w,N\to \infty}A(w,N) = 0
    \end{equation*}
    thanks to \eqref{H2}. Thus
    \begin{equation*}
    \begin{aligned}
    \mathbb{E}\left[\int_0^t\frac{\|\tilde{z}(t)-\tilde{z}(s)\|}{t-s}ds\right] =  \int_0^t\frac{\mathbb{E}\|\tilde{z}(t)-\tilde{z}(s)\|}{t-s}ds\\
        \le \sqrt{2A(w,N)}\int_{0}^t\frac{\sqrt{t-s}}{t-s}ds = \sqrt{2tA(w,N)}.
    \end{aligned}
    \end{equation*}
    Therefore, we have
    \begin{equation}\label{f4}
        \int_0^t\frac{\|\tilde{z}(t)-\tilde{z}(s)\|}{t-s}ds \to 0
    \end{equation}
    in mean, and consequently in probability. Combining this with \eqref{f5}, Lemma \ref{lem:scale}, the convergences \eqref{eps1}, \eqref{eps2}, and \eqref{eps3}, we can apply the Gronwall lemma to finally get
    \begin{equation*}
        \|\tilde{u}(t) - \uu(t)\| \to 0 \quad\text{ in probability},
    \end{equation*}
    which finishes the proof of Theorem \ref{thrm:conv_smooth}.
    
\end{proof}

\section{General first order chemical reaction networks}\label{sec:general_system}
In this section, we show that the techniques presented above can be applied to any first order chemical reaction network with heterogeneities. Let $\S_1, \ldots, \S_K$ be $K\ge 1$ chemicals which react through the following reactions
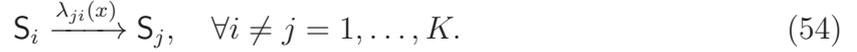
\begin{equation}\label{reaction_network}
    \S_i \xrightarrow{\lambda_{ji}(x)} \S_j, \quad \forall i\ne j = 1,\ldots, K.
\end{equation}
Here $\lambda_{ji}:\Omega \to \R_+$ denotes the reaction constant from $\S_i$ to $\S_j$ which depends on the position $x\in\Omega$.

\medskip
\noindent{\bf Macroscopic PDE model}: Let $u_i(x,t)$ be the concentration of the chemical $\S_i$ at position $x\in\Omega$ and time $t>0$. Assume moreover that $D_i(x)$ is the diffusion coefficient of $\S_i$. The macroscopic reaction-diffusion system reads as
    \begin{equation}\label{general:macro}
        \partial_t \uu_i - \nabla \cdot (D_i(x)\nabla \uu_i) = \sum_{j=1}^{K}\lambda_{ij}(x)\uu_j, \quad x\in\Omega, \; t>0, \quad \forall i=1,\ldots, K,
    \end{equation}
    where 
    \begin{equation}\label{ee}
        \lambda_{ii}(x) = -\sum_{j=1, j\ne i}^{K}\lambda_{ji}(x), \quad \forall i=1,\ldots, K,
    \end{equation}
    subject to the homogeneous Dirichlet boundary conditions
    \begin{equation*}
        \uu_i(x,t) = 0, \quad x\in\partial\Omega,\; t>0,\quad \forall i=1,\ldots, K,
    \end{equation*}
    and non-negative initial data 
    \begin{equation*}
        \uu_i(x,0) = \uu_{i,0}(x), \quad x\in \Omega, \quad \forall i=1,\ldots, K.
    \end{equation*}
    The diffusion coefficients, which could be possibly discontinuous, are merely assumed to be bounded from above and below, i.e.  there $D_*, D^*>0$ such that
    \begin{equation}\label{diff_general}
        D_* \le D_i(x)\le  D^* \quad\forall i=1,\ldots, K, \; \text{a.e. } x\in\Omega,
    \end{equation}
    and the reaction rate functions are bounded, non-negative, i.e.
    \begin{equation}\label{reaction_general}
        \lambda_{ij}\in L^{\infty}(\Omega), \quad \lambda_{ij}(x) \ge 0, \quad \forall i\ne j, \text{ a.e. } x\in\Omega.
    \end{equation}
    \begin{definition}\label{def_weak_reversible}
        The first order chemical reaction network \eqref{reaction_network} is called weakly reversible if for any two species $\S_i$ and $\S_j$, $i\ne j$, there exist a sequence $r_1 = i, r_2, \ldots, r_s = j$, $s\ge 2$, such that all reactions 
        \begin{equation*}
            \S_{r_k} \xrightarrow{\lambda_{r_{k+1}r_k}} S_{r_{k+1}}, \quad k =1, \ldots, s-1
        \end{equation*}
        happen with $\lambda_{r_{k+1}r_k}>0$.
    \end{definition}
    \begin{remark}\hfill
        \begin{itemize}
            \item If we consider the reaction network \eqref{reaction_network} as a directed graph where all the species are nodes and reactions are directed edges, then the weak reversibility in Definition \ref{def_weak_reversible} is equivalent to the strongly connectedness of the corresponding directed graph.
            \item The weak reversibility for general chemical reaction networks is in fact more general, especially for higher order chemical reactions, in which the corresponding directed graph might have several disjoint strongly connected components. However, since the network in \eqref{reaction_network} is of first order, each strongly connected component can be treated separately, and therefore we adapt the definition in Definition \ref{def_weak_reversible} for convenience.
        \end{itemize}
    \end{remark}

\begin{theorem}
    Assume \eqref{diff_general} and \eqref{reaction_general}. Then, for any non-negative, bounded initial data $\uu_0\in \LO{\infty}^K$, there exists a unique global weak solution to \eqref{general:macro}, which is bounded uniformly in time,
    \begin{equation}\label{ee1}
    \limsup_{t\to\infty}\sup_{i=1,\ldots, K}\|\uu_i(t)\|_{\LO{\infty}} \le \varrho <+\infty
    \end{equation}
    for some $\varrho>0$. Moreover, suppose that the network is weakly reversible, and $\lambda_{ij}(x) \equiv \gamma_{ij}\varphi(x) \ge 0$ for all $i\ne j$, where $\varphi(x) \ge  0$ for all $x\in\Omega$, and $\gamma_{ij} > 0$ if $a_{ij} \ne 0$ and $\gamma_{ij} = 0$ otherwise. Then the solution decays exponentially in time, i.e for any $p \in [2,\infty)$ there are $C_p, \alpha_p>0$ such that
    \begin{equation*}
        \sum_{i=1}^K\|\uu_i(t)\|_{\LO{p}} \le C_pe^{-\alpha_p t} \quad \text{ for all } t>0.
    \end{equation*}
\end{theorem}
\begin{proof}
    Denote by $f_i(\uu)$ the right hand side of \eqref{general:macro}. Thanks to \eqref{ee}, it follows that
    \begin{equation*}
        \sum_{i=1}^Kf_i(\uu) = 0.
    \end{equation*}
    The global existence of bounded solutions and uniform-in-time boundedness then follow from \cite{fitzgibbon2021reaction} immediately. It remains to show the exponential decay in the case of weakly reversible networks and $a_{ij}(x) = \gamma_{ij}\varphi(x)$. We rewrite the equation of $\uu_i$ as 
    \begin{equation*}
        \pa_t \uu_i - \na\cdot(D^i(x)\na \uu_i) = \varphi(x)\sum_{j=1}^K\gamma_{ij}\uu_i.
    \end{equation*}
    Due to the weak reversibility of the network, there exists a unique positive state $\uu_{\infty} = (\uu_{i,\infty}) \in (0,\infty)^K$ satisfying that
    \begin{equation*}
        \left\{
        \begin{aligned}
            &\sum_{j=1}^{K}\gamma_{ij}\uu_{i,\infty} = 0, \quad \forall i = 1,\ldots, K,\\
            &\sum_{i=1}^{K}\uu_{i,\infty} = 1,
        \end{aligned}
        \right.
    \end{equation*}
    see e.g. \cite[Lemma 1.7]{fellner2017entropy}. Now we consider the relative energy function
    \begin{equation*}
        \mathcal{E}(\uu) = \sum_{i=1}^K\int_{\Omega}\frac{|\uu_i|^2}{\uu_{i,\infty}}dx.
    \end{equation*}
    Thanks to \cite[Lemma 2.3]{fellner2017entropy}, we have
    \begin{align*}
        \frac{d}{dt}\mathcal{E}(\uu) &= -\sum_{i=1}^K\int_{\Omega}D^i(x)|\na \uu_i|^2dx - \sum_{i<j}\int_{\Omega}\varphi(x)(\gamma_{ij}\uu_{j,\infty}+\gamma_{ji}\uu_{i,\infty})\bra{\frac{\uu_i}{\uu_{i,\infty}} - \frac{\uu_j}{\uu_{j,\infty}}}^2dx\\
        &\le -\sum_{i=1}^{K}D_*\int_{\Omega}|\na \uu_i|^2dx \quad (\text{using (\ref{diff_general})})\\
        &\le -\sum_{i=1}^K D_*C_P\int_{\Omega}|\uu_i|^2dx \quad (\text{using Poincar\'e inequality})\\
        &\le -\delta \mathcal{E}(\uu)
    \end{align*}
    for some $\delta>0$, which implies
    \begin{equation*}
        \mathcal{E}(\uu)(t) \le e^{-\delta t}\mathcal{E}(\uu_0) \quad \text{ for all } t>0,
    \end{equation*}
    and therefore, for all $i=1,\ldots, K$,
    \begin{equation*}
        \|\uu_i(t)\|_{L^2(\Omega)}^2 \le Ce^{-\delta t}, \quad \forall t>0.
    \end{equation*}
    The convergence in $L^p(\Omega)$-norm follows from \eqref{ee1} and $L^2(\Omega)$-convergence
    \begin{equation*}
        \|\uu_i(t)\|_{L^p(\Omega)}^p \le \|\uu_i(t)\|_{L^{\infty}(\Omega)}^{p-2}\|\uu_i(t)\|_{L^2(\Omega)}^2,
    \end{equation*}
    thus
    \begin{equation*}
        \|\uu_i(t)\|_{L^p(\Omega)} \le Ce^{-(2/p)t}, \quad \forall t>0.
    \end{equation*}
\end{proof}

\noindent{\bf Reaction-diffusion master equation}


Let $\Omega = (0,1)^n$ and consider the hypercube with edge length $L > 0$ and volume $V = L^n$. Let $N \in \mathbb{N}$ and consider a uniform partition of $V$ into $N^n$ cells of equal volume $w = L^n/N^n$. Define the $(j_1,\ldots,j_n) \in \{1,\ldots,N\}^n$ cell by $(x_{j_1-1},x_{j_1}]\times\cdots \times (x_{j_n-1},x_{j_n}] = ((j_1 - 1)w,j_1w]\times\cdots\times ((j_n - 1)w,j_nw]$. 

On the boundary of $V$, additional cells are added. Molecules in cells $(j_1,\ldots,j_i,1,j_{i+1},\ldots,j_n)$ and $(j_1,\ldots,j_i,N,j_{i+1},\ldots,j_n)$ may diffuse into these boundary cells in which they are absorbed. That is, the number of molecules in these cells are always zero. The total number of cells in the system is then $(N+2)^n$. 

The number of molecules of each species $l\in \{1,\ldots, K\}$ in cell $\mathbf{j} = (j_1,\ldots,j_n)$ at time $t$ is denoted by $\tilde{X}_{\mathbf{j}}^l(t)$. The number of molecules in cells   $(j_1,\ldots,j_i,0,j_{i+1},\ldots,j_n)$ and $(j_1,\ldots,j_i,N+1,j_{i+1},\ldots,j_n)$ are always zero. The total state of the system is denoted by $\tilde{X}(t)$, which is modeled by a continuous-time Markov jump process on $\mathbb{N}^{K(N+2)^n}$.

Suppose that at time $t$, $\tilde{X}_j^l(t) = k_j^l$. Let $m = (m^1,\ldots, m^K)\in \mathbb{N}^{K(N+2)^n}$  and denote the $j$-th unit vector in $\mathbb{R}^{(N+2)^n}$ by $e_j$. The transition rates for the process $\tilde{X}(t)$ are given by 
\begin{equation*}
	q_{k,k+m} = \left\{
    \begin{aligned}
		&\lambda_{j}^{ll'}k_j^l, && m = \mathbf{1}_j^{l'} - \mathbf{1}_j^l,~j \in \mathcal{J},~l,l' = 1,\ldots,K\\
		&D_j^lN^2k_j^l, && m = \mathbf{1}_{j + e_i}^l - \mathbf{1}_j^l,~j \in \mathcal{J},~l = 1,\ldots,K,~i = 1,\ldots,n\\
		&D_j^lN^2k_j^l, && m = \mathbf{1}_{j-e_i}^l - \mathbf{1}_j^l,~j \in \mathcal{J},~l = 1,\ldots,K,~i = 1,\ldots,n\\
		&D_{j+e_i}^lN^2k_{j+e_i}^l, && m = \mathbf{1}_j^l - \mathbf{1}_{j+e_i}^l,~j \in \mathcal{J}, j_{i}+1 \neq N+1,~l = 1,\ldots,K,~i = 1,\ldots,n\\
		&D_{j-e_i}^lN^2k_{j-e_i}^l, && m = \mathbf{1}_{j}^l - \mathbf{1}_{j-e_i}^l,~j \in \mathcal{J},~j_{i}-1 \neq 0,~l = 1,\ldots,K,~i = 1,\ldots,n\\
		&0, && \mathrm{otherwise}.
    \end{aligned}
    \right.
\end{equation*}


Let $\mathcal{J} = \{1,\ldots,N\}^n$. In this setting, the RDME is then given by
\begin{align}
	\frac{\mathrm{d}}{\mathrm{d}t} P(k,t) &= \sum_{l = 1}^K \sum_{j \in \mathcal{J}} \sum_{i=1}^n \left(q_{k-\mathbf{1}_j^l + \mathbf{1}_{j + e_i}^l,k}P(k-\mathbf{1}_{j}^l + \mathbf{1}_{j+e_i}^l,t) - q_{k,k-\mathbf{1}_j^l + \mathbf{1}_{j + e_i}^l} P(k,t)\right.\notag\\
	&\left. ~~~ + q_{k-\mathbf{1}_{j}^l + \mathbf{1}_{j-e_i}^l,k}P(k-\mathbf{1}_{j}^l + \mathbf{1}_{j + e_i}^l,t) - q_{k,k-\mathbf{1}_j^l + \mathbf{1}_{j-e_i}^l}P(k,t)\right)\notag\\
	&~~~ + \sum_{l=1}^K \sum_{j \in \mathcal{J}} q_{k - \mathbf{1}_j^l + \mathbf{1}_{j}^{3-l},k} P(k-\mathbf{1}_{j}^l + \mathbf{1}_j^{3-l},t) - q_{k,k - \mathbf{1}_j^l + \mathbf{1}_j^{3-l}}P(k,t). 
\end{align}
Here $\mathbf{1}_{j'}^{l'} \in \mathbb{N}^{K(N+2)^n}$, $l' \in \{1,\ldots,K\}$, $j'\in \{0,\ldots,N+1\}^n$ satisfies 
\begin{align}
	\mathbf{1}_{j}^l(j,l) = \left\{\begin{array}{ll}
		1, & j = j',~l' = l\\
		0, & j \neq j'~or~l'\neq l
	\end{array}\right.,~j' \in \{0,\ldots,N+1\}^n,~l' \in \{1,\ldots,K\}.
\end{align}


\noindent{\bf Macroscopic limit}: Define the concentration process $C^l_{\mathbf{j}}(t) = \tilde{X}^l_{\mathbf{j}}(t)/w^n$, and $C^l_{\mathbf{j}}(t) = 0$ if $\mathbf{j}$ is the artifical boundary cell. By defining $\mathcal{X}_N = L_2[0,1]^n$ and
\begin{equation*}
    u_l(x,t) = \sum_{\mathbf{j}}C^l_{\mathbf{j}}(t)\mathbf{1}\left\{x\in \prod_{k=1}^n\bigg[\frac{(j_k-1)w}{V}, \frac{j_kw}{V}\bigg) \right\}, \quad l=1,\ldots, K.
\end{equation*}
For $\vartheta > \varrho$, with $\varrho$ in Theorem \ref{ee1}, we define $\tau_{S_{\vartheta}}$ as the first exit time of $u(t)$ from 
\begin{equation*}
    S_{\vartheta} = \left\{u\in \mathcal{X}_N^+: \sup_{0\le x\le 1}\sum_{k=1}^Ku_k(x) \le \vartheta \right\}.
\end{equation*}
Finally, we define the stopped process $\tilde{u}(t)$ as
\begin{equation*}
    \tilde{u}(\cdot,t) = u(\cdot,\min\{t,\tau_{S_{\vartheta}} \}).
\end{equation*}
By repeating the arguments in Section \ref{sec:two_chemicals} we obtain the following result.

\begin{theorem}
    Assume \eqref{diff_general}	and \eqref{reaction_general}, and additionally $D_i \in C^1(\Omega) \cap C(\overline{\Omega})$ for all $i=1,\ldots, K$. Further, assume that the following conditions hold: 
	\begin{enumerate}
		\item[{\normalfont (H1')}] $\|\tilde{u}(0) - \mathbf{u}(0)\| = 0$ in probability as $N,w \rightarrow \infty$;
		\item[{\normalfont (H2')}] and
		\begin{align*}
			\frac{N^2}{w^n} \rightarrow 0 \quad \text{ as } \quad N,w \rightarrow \infty.
		\end{align*}
	\end{enumerate}
    Then, for any $\delta > 0$ and $T \in (0,\infty)$, 
	\begin{align*}
			\sup_{t \in [0,T]} \mathbb{P}(\|\tilde{u}(t) - \mathbf{u}(t)\| > \delta) \rightarrow 0 \quad \text{ as } \quad N,w \rightarrow \infty.
		\end{align*}
\end{theorem}

We remark that the main difference from the case of $K = 2$ chemical species arises in the analog of Lemma~\ref{lem:scale}. In the case of general $K$, for the martingale process $\tilde{z}(t)$ we have
\begin{align}
	\mathbb{E}_{\tilde{u}}[\|\tilde{z}(t)\|^2] &= \int_0^t \mathbb{E}_{\tilde{u}}\left[\tilde{\tau}(\tilde{u}(s))\left(\sum_m \left\|\frac{m}{w^n}\right\|\tilde{\sigma}\left(\tilde{u}(s),\frac{m}{w^n}\right)\right)\right]\mathrm{d}s\notag\\
	&= \int_0^t \mathbb{E}_{\tilde{u}}\left[\sum_m \left\|\frac{m}{w^n}\right\|^2w^n f(\tilde{u}(s),m)\right]\mathrm{d}s\notag\\
	&= \frac{1}{w^nN^n} \int_0^t \mathbb{E}_{\tilde{u}}\left[\sum_{m} |m|^2 f(\tilde{u}(s),m)\right]\mathrm{d}s\notag\\
	&= \frac{1}{w^nN^n} \int_0^t \mathbb{E}_{\tilde{u}}\left[\sum_m |m|^2 f(\tilde{u}(s),m)\right]\mathrm{d}s\notag\\
	&= \frac{1}{w^nN^n} \int_0^t \mathbb{E}_{\tilde{u}}\left[\sum_m |m|^2 q_{k,k+m}(w^n\tilde{u}(s))w^{-n}\right]\mathrm{d}s\notag\\
	&= \frac{1}{w^nN^n} \mathbb{E}_{\tilde{u}}\left[\sum_{l=1}^K \sum_{l'= 1}^K \sum_{j \in \mathcal{J}} \lambda_j^{ll'} \tilde{u}^l_j(s)\right.\notag\\
	&~~~\left. + \sum_{l=1}^K \sum_{j \in \mathcal{J}} \sum_{i=1}^n 2D_j^lN^2\tilde{u}_j^l(s) + D_{j + e_i}^lN^2\tilde{u}_{j+e_i}^l(s) + D_{j - e_i}^lN^2\tilde{u}_{j-e_i}^l(s)\right]\mathrm{d}s\notag\\
	&\leq \frac{1}{w^nN^n} \int_0^t \left(K^2N^n\lambda^*\rho + 4KnN^{2+n}D^*\rho\right)\mathrm{d}s\notag\\
	&= \frac{t\rho}{w^n}\left(K^2\lambda^* + 4KnD^*N^2\right)
\end{align}	
where $\lambda^* = \max_{i,j}\|\lambda_{ij}\|_{L^{\infty}(\Omega)}$. Hence, if we seek $\mathbb{E}[\|\tilde{z}(t)\|^2] \rightarrow 0$ as $N,w \rightarrow 0$, we require $\dfrac{N^2}{w^n} \rightarrow 0$, which corresponds to assumption (H1'). 


\medskip
\noindent{\bf Acknowledgement.}
This research was funded by the Austrian Agency for International Cooperation in Education and Reseach (OeAD-GmbH), project number FR 03/2020. This research was also funded in part by the Campus France PHC AUTRICHE - AMADEUS programme, project number 44081QB.

\medskip
The research was carried out during the visit of the first author to the University of Graz, and of the second author to University of Lyon, and both universities' hospitality is greatly acknowledged.

\medskip

\newcommand{\etalchar}[1]{$^{#1}$}

\end{document}